\documentclass{amsart}
\usepackage{amssymb}
\newtheorem{Theo}{Theorem}
\newtheorem{Lem}{Lemma}
\newcommand{\Z}{\mathbb{Z}}
\begin{document}
\title[Automaticity of Primes]{The Automaticity of the Set of Primes}
\author[T. Dubbe]{Thomas Dubbe}
\begin{abstract}
The automaticity $A(x)$ of a set $\mathcal{X}$ is the size of the smallest automaton that recognizes $\mathcal{X}$ on all words of length $\leq x$. We show that the automaticity of the set of primes is at least $x\exp\left(-c(\log\log x)^2\log\log\log x\right)$, which is fairly close to the maximal automaticity.
\end{abstract}
\maketitle
\section{Introduction and Results}
A finite automaton over an alphabet $\mathcal{A}$ consists of a finite set $S$ of states, a starting state $s_0$ in $S$, a transition function $\delta:A\times S\rightarrow S$, and a set $T\subseteq S$ of accepting states. A word $w=a_1a_2\dots a_n$ over $\mathcal{A}$ induces the sequence of states $s_0, \ldots, s_n$, where $s_{i+1}=\delta(a_i, s_i)$, and $w$ is accepted if $s_n\in T$. A language $L\subseteq\mathcal{A}^*$ is called automatic, if there exists a finite automaton that accepts exactly the elements of $L$. For more background on finite automata we refer the reader to the book by Allouche and Shallit \cite{AS}.

If $\mathcal{A}=\{0, 1, \ldots, q-1\}$, then we can identify words over $\mathcal{A}$ with integers by means of the $q$-adic representation. Studying number theoretic properties of automatic sets of integers is a topic with quite some history. Hartmanis and Shanks \cite{HS} showed that the set of primes is not automatic. This leads to the question how far the set of primes deviates from an automatic set. One way to measure this is automaticity: We define the automaticity of a set $\mathcal{X}$ as the function $A(x)$ that associates to an integer $x$ the size of the smallest automaton accepting an integer $n\leq x$ if and only if $n\in\mathcal{X}$. Clearly, the automaticity of any set satisfies $A(x)\leq x$, and a set is automatic if and only if $A(x)=\mathcal{O}(1)$. Here we show that the set of primes is close to the maximal automaticity.

\begin{Theo}
There exists a constant $c$ such that the set of primes has automaticity at least 
\[x\exp\left(-c(\log\log x)^2\log\log\log x\right)\]
where $x$ is sufficiently large.
\end{Theo}
We determine the number of states an automaton needs to correctly identify the prime numbers below $q^n$. 
Let $m<n$ be chosen later. We will show that a significant proportion of the $q^{n-m}$  words $w$ of length $n-m$ have the property that the set $A_w=\{a: 1\leq a<p^m, ap^{n-m}+w\mbox{ prime}\}$  essentially determines $w$. In particular, there is a large set $\mathcal{W}$ such that for $w, w'\in\mathcal{W}$ with $w\not= w'$ we have $A_w\neq A_{w'}$.  Now $|\mathcal{W}|$ is a lower bound for the size of an automaton recognizing primes up to $q^n$, since if we read two words $w,w'\in\mathcal{W}$, and reach states $s_w, s_{w'}$, then the states $s_w, s_{w'}$ determine the sets $A_w, A_{w'}$. Since the latter are different, so are the states, that is, $|\{s_w:w\in\mathcal{W}\}|=|\mathcal{W}|$ is a lower bound for the size of an automaton recognizing primes up to $q^n$.\\
I would like to thank Prof. Schlage-Puchta, who made me aware of this problem.

\section{Some number theory}

We will use the following consequence of the large sieve:

\begin{Lem} 
\label{Lem:Siebert}
Let $\mathcal{A}\subseteq[1, x]$ be a set of integers and $\Omega_p$ be sets of residue classes $(\bmod{p})$ for every prime. If we set $\omega(p)=|\Omega_p|$ and assume that $\omega(p) = k$ for all $p$ except for some finite cases where it is less than $k$, then 
\begin{align*}
\#\{a\in\mathcal{A}\;|\; (a\bmod{p})\not\in\Omega_p \;\text{for all}\;p\leq z\}\leq\\
2^kk!\prod_p\left(1-\frac{\omega(p)}{p}\right)\left(1-\frac{1}{p}\right)^{-k}\frac{x}{\log^kx}\left(1-\frac{r_k}{\log x}\right)^{-1}.
\end{align*}
for $x>e^{r_k}$, $z=\sqrt{\frac 2 3 x}$, $r_k = \mathcal{O}{(k^2\log k)}$ and $k\geq 2$
\end{Lem}

The case $k=2$ was essentially proven by Siebert\cite{Siebert}. Although he formulates his result only for the case of prime pairs $p, ap+b$, the relevant computation of the sum $L$ only makes use of the number of residue classes avoided, and applies to the above statement as well. Selberg's sieve can be used to replace the constant 16 by $8+\epsilon$, however, the above has the advantage of being completely uniform. The remaining cases are Theorem 2 of \cite{Dubbe}.
Put $x=q^n$, $y=q^m$. 
\begin{Lem}
\label{Lem:non unique}
Let $k\geq2$.
Suppose that $w_1,\ldots,w_k$ are different words of length $\leq n-m$, such that $A_{w_i}=A_{w_j}$ for $i\not=j$ and $(q,w_i)=1$. Then 
\[ |A_{w_i}| \leq 3\cdot2^kk!\prod_p\left(1-\frac{\omega(p)}{p}\right)\left(1-\frac{1}{p}\right)^{-k}\frac{y}{\log^ky}\left(1-\frac{r_k}{\log y}\right)^{-1},\]
for $y>e^{2r_k}$, where $\omega(p)$ is the number of residue classes modulo $p$, which contain at least one $w_i\;(\bmod{\;p})$.
\end{Lem}

\begin{proof}
Let $\mathbb{P}$ be the set of prime numbers in the range $[1, x]$.

First, we note that an integer $a$ fulfills the condition $a(x/y)+w_i\in\mathbb{P}$ if and only if it satisfies $a(x/y)+w_j\in\mathbb{P}$ for every other $j\not= i$. This is due to the assumption that $A_{w_i}=A_{w_j}$ for $i\not=j$. So we have 
\[A_{w_i} = \{1\leq a <y\;|\; a(x/y)+w_i\in\mathbb{P}\;\text{for}\;i=1,\ldots,k\}.\]
We set $z=\sqrt{2/3y}$ and $y_0 = z\cdot y/x\leq z$. Suppose we have an integer $y_0 < a <y$ so that $q = a(x/y) + w_i\in\mathbb{P}$. Such $q$ lies in the range $(z,x]$ and is obviously coprime to every prime $p\leq z$. Therefore, $q \not\equiv 0\;(\bmod{\;p})$ and the size of the set $A_{w_i}$ is bounded by $y_0 + \#\{y_0\leq m< y\;|\;\left[m\right]_p\not\in\Omega_p\;\text{for all}\; p\leq z\}$, where $\Omega_p$ contains all solutions $m\;(\bmod{\;p})$ of $m(x/y)+w_i\equiv 0\;(\bmod{\;p})$ for $i=1,\ldots,k$. Of course, we have at most $k$ different solutions modulo $p$ and hence $\omega(p)\leq k$. On the other hand, we see that a solution modulo $p$ appears twice iff we have $w_i\equiv w_j\;(p)$ or $p|w_i-w_j$, respectively. This is only possible for small primes $p<x/y$, since all $w_i$ are words of length $\leq n-m$ and therefore $0 < |w_i-w_j| < x/y$. Consequently, the assumptions of Lemma~\ref{Lem:Siebert} are satisfies and we have
\begin{align*}
|A_{w_i}| &\leq y_0 + 2^kk!\prod_p\left(1-\frac{\omega(p)}{p}\right)\left(1-\frac{1}{p}\right)^{-k}\frac{y}{\log^ky}\left(1-\frac{r_k}{\log y}\right)^{-1}\\
          &\leq 2^kk!\prod_p\left(1-\frac{\omega(p)}{p}\right)\left(1-\frac{1}{p}\right)^{-k}\frac{y}{\log^ky}\left(\frac{\sqrt{\frac{2}{3}}}{2^kk!W_K}\frac{\log^k y}{\sqrt{y}}+\frac{\log y}{r_k-\log y}\right)
\end{align*}
for $y>e^{r_k}$, where $W_k$ is the product over primes. Using an argument similar to that in section $5$ of \cite{Dubbe}, one can show that the term in brackets is bounded by $3$ for $y>e^{2r_k}$.
\end{proof}
 
The following is a reformulation of the Brun-Titchmarsh inequality (see \cite{BT}).
\begin{Lem}
\label{Lem:BT}
We have $|A_w|<\frac{2x}{\varphi(q^{n-m})\log(y)} = 2\frac{q}{\varphi(q)}\frac{y}{\log (y)}$ for $(q,w)=1$.
\end{Lem}

Now, every prime $x/y<p<x$ gives rise to exactly one $a$ in exactly one $A_w$, hence, by the prime number theorem $\sum_{w<x/y}|A_w|\geq \sum_{x/\log(x) < p < x} 1\sim\frac{x}{\log x}$, provided that $x\rightarrow\infty$ and $y\geq \log x$. Next, we group together all $A_w$ which account for exactly $k$ different words with $1\leq k < K$ and also those which account for at least $K$ words. Let $N_k$ be the number of the former, $R_K$ the number of the latter sets. Here $K\geq 2$ is a natural number that we will chose later. Note that the number of different sets $A_w$ is bounded below by $N=\sum_{k<K} N_k$, and bounded above by the size of an automaton recognizing the primes in $[1, q^n]$, so it remains to estimate $N$. Note further that we have $A_w \not= \emptyset$ only for words $w$ with $(q,w)=1$, otherwise the prime condition defining this set is never satisfied. We are using Lemma~\ref{Lem:non unique} and Lemma~\ref{Lem:BT} to estimate the size of our sets $A_w$ and obtain
\begin{align}\label{EQ1}
\frac{1}{2}\frac{x}{\log x} & \leq \sum_{k=1}^{K-1} C_kN_k\frac{y}{\log^ky} + C_KR_K\frac{y}{\log^Ky}\notag\\
 & \leq N\left(\sum_{k=1}^{K-1} C_k\frac{y}{\log^ky}\right) + C_K\frac{x}{\log^Ky},
\end{align}
with $C_1 = 2\frac{q}{\varphi(q)}$ and
\[ C_k = 3\cdot2^kk!\prod_p\left(1-\frac{\omega(p)}{p}\right)\left(1-\frac{1}{p}\right)^{-k}\]
for $k\geq 2$, provided that $y>e^{2r_K}$ and $x\geq x_0$ being large enough.
\begin{Lem}\label{const}
For $k\geq 2$ there exist constants $D_1$ and $D_2$ which depend at most on $q$, so that
\[C_k\leq D_1\frac{q}{\varphi(q)}(D_2k\log k)^k\left(\log\log\frac{x}{y}\right)^{k-1}\]
provided that $\frac{x}{y} > \max\{k,e^e\}$ and $\log\frac{x}{y} > q$.
\end{Lem}
\begin{proof}
First, consider the set of equations $m(x/y) + w_i\equiv 0\;(p)$ for $i=1,\ldots,k$ and with $x/y = q^{n-m}$. Remember that $\omega(p)$ is the number of different solutions $(\bmod{\;p})$ and that we only consider words $w_i$ with $(q,w_i)=1$. Therefore, we have $\omega(p) = 0$ for $p\mid q$ and $\omega(p)\geq 1$ otherwise, since we have no solution in the former case, but at least one in the latter. We also saw in the proof of lemma \ref{Lem:non unique} that we have $\omega(p) = k$ for $p>x/y$.  This gives us
\begin{align}\label{EQ2}
C_k=   A_k &\prod_p\left(1-\frac{\omega(p)}{p}\right)\left(1-\frac{1}{p}\right)^{-k}\notag\\
   =   A_k &\prod_{p\mid q}\left(1-\frac{1}{p}\right)^{-k}\prod_{p\nmid q}\left(1-\frac{\omega(p)}{p}\right)\left(1-\frac{1}{p}\right)^{-k}\notag\\
   \leq A_k& \frac{q}{\varphi(q)}\prod_{p\leq k}\left(1-\frac{1}{p}\right)^{-(k-1)}\prod_{\substack{p > k\\p\mid q}}\left(1-\frac{1}{p}\right)^{-(k-1)}\notag\\
	         & \prod_{\substack{k<p\leq \frac{x}{y}\\ p\nmid q}}\frac{p-\omega(p)}{p-k}\prod_{\substack{p>k \\ p\nmid q}}\left(1-\frac{k}{p}\right)\left(1-\frac{1}{p}\right)^{-k}\notag\\
	 =    A_k& \frac{q}{\varphi(q)}\prod_{p\leq k}\left(1-\frac{1}{p}\right)^{-(k-1)}\prod_{p>k}\left(1-\frac{k}{p}\right)\left(1-\frac{1}{p}\right)^{-k}\notag\\
	         & \prod_{\substack{k<p\leq \frac{x}{y}\\ p\nmid q}}\frac{p-\omega(p)}{p-k}	\prod_{\substack{p > k\\ p\mid q}}\frac{p-1}{p-k}		
\end{align}
with $A_k = 3k!2^k$. Next, we treat the product over primes in the range $k<p\leq\frac x y$ with $p\nmid q$. Clearly, one factor is $>1$ if we have $\omega(p) \neq k$ for the corresponding prime $p$. Comparing the $n$-th largest prime $p_n$ with this property to the $n$-th largest prime $p'_n$ greater than $k$ and with $p\nmid q$ gives us
\[\frac{p_n-\omega(p)}{p_n-k}\leq \frac{p_n-1}{p_n-k}\leq \frac{p'_n-1}{p'_n-k}.\]
Writing $\mathcal{P} = \{k<p\leq x/y\;\mid\;\omega(p)\neq k\;\text{and}\;p\nmid q\}$ and $\mathcal{P}' = \{k < p \leq z_0\;\mid\;p\in\mathbb{P}\;\text{and}\;p\nmid q\}$ with $|\mathcal{P}|= |\mathcal{P}'|$ leads us further to
\[
\prod_{\substack{k<p\leq \frac{x}{y}\\ p\nmid q}}\frac{p-\omega(p)}{p-k} = \prod_{p\in\mathcal{P}}\frac{p-\omega(p)}{p-k} \leq \prod_{p\in\mathcal{P}'}\frac{p-1}{p-k} = \prod_{\substack{k<p\leq z_0\\ p\nmid q}}\frac{p-1}{p-k}.
\]
On the other hand, we have $\omega(p)\neq k$ only if there are words $w_i$ and $w_j$ such that $p\mid w_i-w_j$. Therefore, the product on the left ranges over primes that divide these kind differences. If we set $P_{ij} = \prod_{p\mid w_i-w_j}p$, then $P_{ij} <x/y$ and 
\[
\prod_{\substack{k<p\leq z_0\\ p\nmid q}}p = \prod_{p\in\mathcal{P}'}p \leq \prod_{p\in\mathcal{P}}p\leq \prod_{i<j}P_{ij} < \left(\frac{x}{y}\right)^{k^2}.
\]
Using the prime number theorem in the form $\sum_{p\leq z}\log p \sim z$ for $z\rightarrow\infty$, we get $z_0 < k^2\log(x/y) + k +c$, where $c$ depends at most on $q$. In combination with the condition $\log(x/y)\geq q$ we deduce for (\ref{EQ2}) the following
\begin{align*}
C_k \leq    A_k& \frac{q}{\varphi(q)}\prod_{p\leq k}\left(1-\frac{1}{p}\right)^{-(k-1)}\prod_{p>k}\left(1-\frac{k}{p}\right)\left(1-\frac{1}{p}\right)^{-k}\notag\\
	             & \prod_{k<p\leq k^2\log\frac{x}{y}+k+c}\frac{p-1}{p-k}\\
		=   		A_k& \frac{q}{\varphi(q)}\prod_{p\leq k^2\log\frac{x}{y}+k+c}\left(1-\frac{1}{p}\right)^{-(k-1)}\prod_{p>k^2\log\frac{x}{y}+k+c}\left(1-\frac{k}{p}\right)\left(1-\frac{1}{p}\right)^{-k}\\
		\leq    A_k& \frac{q}{\varphi(q)}\prod_{p\leq k^2\log\frac{x}{y}+k+c}\left(1-\frac{1}{p}\right)^{-(k-1)}
\end{align*}
Now we have $\prod_{p\leq z}\left(1-\frac{1}{p}\right)^{-1} = \mathcal{O}(\log(z))$ according to mertens' theorem and $A_k = \mathcal{O}((c'k)^k)$ according to sterling's formula. The implied constants are in both cases absolute. Thus we have
\begin{align*}
C_k \leq& D'_1\frac{q}{\varphi(q)}(D_2'k)^k\left(\log(k^2\log\frac{x}{y}+k+c)\right)^{k-1}\\
    \leq& D'_1\frac{q}{\varphi(q)}(D_2'k)^k\left(D'_3\log k\log(\log\frac{x}{y})\right)^{k-1}\\
		\leq& D_1(D_2 k\log k)^k\left(\log\log\frac{x}{y}\right)^{k-1}
\end{align*} 
for $\log\log x/y \geq1$.
\end{proof}

\section{conclusion of the proof}

Lemma~\ref{const} allows us to estimate (\ref{EQ1}) by
\begin{align}\label{EQ3}
N \geq \frac{x}{y}\left(\frac{1}{2\log x} - \frac{E(K)\left(\log\log\frac{x}{y}\right)^{K-1}}{\log^Ky}\right)\left(\sum_{k=1}^{K-1}E(k)\frac{\left(\log\log\frac{x}{y}\right)^{k-1}}{\log^ky}\right)^{-1}
\end{align}
for $x\geq x_0$, with the constant
\[ E(k) = D_1(C_0D_2k\log (k+1))^k \]
(the parameter $C_0\geq 1$ will be selected later) and under the conditions
\begin{align}\label{cond}
y > e^{2r_K},\quad y > \log x,\quad \frac x y > \max\{K,e^e\},\quad \log\frac x y > q.
\end{align}
We choose $y = q^m$ with $m = \left\lceil \frac{\log y_0}{\log q}\right\rceil$, where $y_0 = y_0(x)$ is the solution of
\[ \log^K y_0 = 4E(K)\log x\left(\log\log x\right)^{K-1}\]
for a given $x$. Then we have $qy_0 > y \geq y_0$ and $\log^Ky \geq 4E(K)\log x\left(\log\log x\right)^{K-1}$, which gives us the following for (\ref{EQ3})
\begin{align*}
N &\geq \frac{x}{y}\frac{1}{4\log x}\left(\sum_{k=1}^{K-1}E(k)\frac{\left(\log\log\frac{x}{y}\right)^{k-1}}{\log^ky}\right)^{-1}\\
  &= \frac{x}{y}\left(\sum_{k=1}^{K-1}4E(k)\frac{\log x\left(\log\log\frac{x}{y}\right)^{k-1}}{\log^ky}\right)^{-1}.
\end{align*}
Furthermore, our choice of $y$ enables us to estimate the terms in the sum on the right through
\begin{align*}
4E(k)&\frac{\log x(\log\log\frac{x}{y})^{k-1}}{\log^k y} \leq 4E(k) \frac{\log x}{\log\log x}\left(\frac{\log\log x}{\log y}\right)^k\\
&\leq 4E(k)\frac{\log x}{\log\log x}\left(\frac{\log\log x}{4E(K)\log x}\right)^{\frac k K } \\
&= (4D_1)^{1-\frac k K}\left(\frac{k\log(k+1)}{K\log(K+1)}\right)^k\frac{\log x}{\log\log x}\left(\frac{\log\log x}{\log x}\right)^{\frac k K }\\
&< 4D_1 \frac{\log x}{\log\log x}\left(\frac{\log\log x}{\log x}\right)^{\frac k K }.
\end{align*}
Therefore, the sum is bounded by
\begin{align*}
\sum_{k=1}^{K-1}4E(k)\frac{\log x\left(\log\log\frac{x}{y}\right)^{k-1}}{\log^ky} < E_0 K \left(\frac{\log x}{\log\log x}\right)^{1-\frac{1}{K}}
\end{align*}
with $E_0 = 4D_1$. Combing this with 
\[ y\leq q y_0 = q\exp\left(\left(E(K)\frac{\log x}{\log\log x}\right)^{\frac 1 K}\log\log x\right) \]
leads us to 
\begin{align}\label{lastEQ}
N > \frac{q}{E_0K}\left(\frac{\log\log x}{\log x}\right)^{1-\frac 1 K} x \exp\left(-\left(E(K)\frac{\log x}{\log\log x}\right)^{\frac 1 K}\log\log x\right)
\end{align}
for $x\geq x_0$ and under the conditions (\ref{cond}). Finally, we choose
\[ K = \left\lceil \log\log x\right\rceil.\]
with $x_0\geq\exp(\exp(2))$ such that $K\geq 2$. Of course, we must make sure that this choice does not contradict our other requirements. First, we notice that $y > e^{2r_K}$ follows if the inequality 
\[\left(\frac{\log x}{\log\log x}\right)^{\frac 1 K}\log\log x > \frac{2r_K}{E(K)^{\frac 1 K}}=\frac{2r_K}{(D_1(C_0 D_2K\log(K+1))^K)^{\frac 1 K}}\]
holds for $x\geq x_0$. This is certainly true, if the parameter $C_0$ is sufficiently large. The other conditions are straight forward, provided that $x_0$ is large enough. Inserting our choice into (\ref{lastEQ}) gives us
\[ N > C\frac{x}{\log x}\exp\left(-C'(\log\log x)^2\log\log\log x \right) \]
for $x\geq x_0$ and our Theorem follows. 

\section{Further results}

A similar, but much easier approach works for the set of squares.
\begin{Theo}
Let $q$ be an odd prime. Then the $q$-automaticity  of the set of squares is $\asymp x^{1/2}$.
\end{Theo}
\begin{proof}
For the lower bound put $n=2m$, and consider an initial word $w$ of length $m$ coprime to $q$. If $a\in A_w$, then $a=u^2\equiv w\pmod{q^m}$. The multiplicative group $(\Z/q^m\Z)^*$ is cyclic, hence, there are exactly two choices for $u$. Each choice of $u\pmod{q^m}$ leads to exactly one element in $A_w$, as $(u\bmod q^m)^2<q^{2m}$, $(u\bmod{q^m}+q^m)^2>q^{2m}$. We conclude that $|A_w|\leq 2$.

Now consider the squares in the range $[q^{n-1}, q^n-1]$ coprime to $q$. There are $q^{n/2}-q^{(n-1)/2}+\mathcal{O}(1)$ squares of this type. Each of these squares is contained in some $A_w$, and each $|A_w|$ is bounded, hence, we can pick a set $\mathcal{W}$ of words, such that for $w\neq w'\in\mathcal{W}$ we have $A_w\neq A_{w'}$, for all $w\in\mathcal{W}$ there exists a square in $[q^{(n-1)}, q^n-1]$ ending with $w$, all squares in this interval are convered in this way, and $|\mathcal{W}|>cq^m$.

For the upper bound we first construct an automaton recognizing the squares not divisible by $q$. If $w$ is a word of length $m$, not beginning with 0, which is a quadratic residue, we have $|A_w|=2$, and if $u\in[1, q^m-1]$ satisfies $u^2\equiv w\pmod{q^m}$, then $A_ww=\{u^2, (q^m-u)^2\}$. We have $(q^m-u)^2 = q^{2m}-2q^mu+u^2$, hence, the $m+1$-st digit of $(q^m-u)^2$ coincides with the $m+1$-st digit of $u^2$ if and only if the first digit of $2u$ is 0. But this contradicts the assumption that $w$ is not divisible by $q$. We conclude that for all words of length $m+1$ we have $|A_w|\leq 1$. We now encode the words of length $m+1$ into a binary tree, and remove all leaves with $A_w=\emptyset$. We now encode all possible extensions, i.e. words of length $\leq m-1$ into a second tree, where we reverse the transition map and declare the root as unique accepting state. In other words, we create an automaton such that starting in the state representing the integer $a$ leads to an accepting state if $a$ is read, and to a fail state otherwise. We now identify each leaf $w$ of the first tree with the unique node $a$ such that $A_w=\{a\}$ of the second tree, and obtain an automaton recognizing squares coprime to $q$. We now add three states dealing with integers divisible by $q$.
\end{proof}
In the reverse direction we have the following.

\begin{Theo}
The automaticity of the set of primes is at most $\mathcal{O}\left(\frac{x}{\log x\log\log\log  x}\right)$.
\end{Theo}
\begin{proof}
We construct an automaton that recognizes all primes of length $\leq n$ as follows. Let $m$ be a parameter to be chosen later. We read the first $n-m$ digits using a $q$-ary tree. Let $\mathcal{A}$ be a complete list of all $A_w$, $w\leq q^{n-m}$. Then we use for each element in this list an automaton with $q^m$ states and obtain that the automaticity is $\mathcal{O}(q^{n-m}+q^m|\mathcal{A}|)$. The trivial bound $|\mathcal{A}|\leq 2^{q^m}$ leads to the bound $\frac{x}{\log x}$, which is valid for any set. To improve this bound we have to give a non-trivial bound for the number of possible sets $A_w$. 

Let $y<q^{n-m}$ be determined later. Let $Q$ be the product of all prime numbers $p<y$ which do not divide $q$. Then each $A_w$ is contained in a set that can be shifted into a set of integers coprime to $Q$. There are $Q$ possible shifts. The number of possible subsets of integers coprime to $Q$ is bounded by $2^{\left(\lfloor\frac{q^m}{Q}\rfloor+1\right)\varphi(Q)}$. We conclude that the automaticity is bounded by
\[
q^{n-m}+q^m 2^{\left(\lfloor\frac{q^m}{Q}\rfloor+1\right)\varphi(Q)}
\]
We choose $y$ maximal subject to the condition $Q<q^m$. By the prime number theorem we can guarantee $Q>q^{m/2}$. Then the automaticity is bounded by
\[
q^{n-m}+q^m 2^{2q^m\frac{\varphi(Q)}{Q}}\leq q^{n-m} + 2^{\mathcal{O}\left(\frac{q^m}{\log\log Q}\right)}\leq q^{n-m} + 2^{\mathcal{O}\left(\frac{q^m}{\log\log q^m}\right)}.
\]
The second summand is negligible, provided that $\frac{q^m}{\log\log q^m}<cn$ for some small constant $c$. Thus, we can take $q^m\asymp n\log\log n$, and our result follows.
\end{proof}

\end{document}